\documentclass{amsart}
\usepackage{epsf}
\usepackage{graphics}
\usepackage{amsmath}
\usepackage{verbatim}

\usepackage{amsfonts}
\usepackage{amssymb}
\usepackage{tikz}
\newtheorem{thm}{Theorem}[section]
\newtheorem{lem}[thm]{Lemma}
\newtheorem{prop}[thm]{Proposition}
\newtheorem{cor}[thm]{Corollary}

\theoremstyle{definition}
\newtheorem{defn}[thm]{Definition}

\newtheorem{rem}[thm]{Remark}

\newtheorem*{thmA}{Theorem \hypertarget{thmA}{A}}
\newtheorem*{thmB}{Theorem \hypertarget{thmB}{B}}
\newtheorem*{thmC}{Theorem \hypertarget{thmC}{C}}

\usepackage{color}
\usepackage{hyperref}
\hypersetup{colorlinks=True}

\providecommand{\ben}{\begin{enumerate}}
\providecommand{\een}{\end{enumerate}}
\providecommand{\bit}{\begin{itemize}}
\providecommand{\eit}{\end{itemize}}

\providecommand{\bc}[2]{{#1\choose#2}}
\providecommand\gbin[2]{\genfrac{[}{]}{0pt}{}{#1}{#2}}

\providecommand{\floor}[1]{\left\lfloor #1 \right\rfloor}
\providecommand{\ceil}[1]{\left\lceil #1 \right\rceil}
\providecommand{\bool}[1]{\mathbf{2}^{[#1]}}
\providecommand{\size}[1]{\left\lvert {#1} \right\rvert}

\providecommand{\F}{\mathbb{F}}
\providecommand{\Lnq}[2]{\mathcal{L}_{#1}(#2)}
\providecommand{\Aa}{\mathcal{A}}
\providecommand{\Bb}{\mathcal{B}}
\providecommand{\Dd}{\mathcal{D}}
\providecommand{\Ee}{\mathcal{E}}
\providecommand{\Ff}{\mathcal{F}}

\providecommand{\Mm}{\mathcal{M}}
\providecommand{\Ss}{\mathcal{S}}
\providecommand{\Xx}{\mathcal{X}}
\providecommand{\Yy}{\mathcal{Y}}
\providecommand{\shade}{\bigtriangledown}
\providecommand{\shadow}{\bigtriangleup}

\newcommand{\twochain}{\raisebox{-2pt}{\begin{tikzpicture} \fill (0,0) circle (1.2pt) (0,8pt) circle (1.2pt); \draw (0,0pt)--(0,8pt);\end{tikzpicture}}}
\newcommand{\threechain}{\raisebox{-6pt}{\begin{tikzpicture} \fill (0,0) circle (1.2pt) (0,8pt) circle (1.2pt) (0,16pt) circle (1.2pt); \draw (0,0pt)--(0,8pt)--(0,16pt);\end{tikzpicture}}}
\newcommand{\Yposet}{\raisebox{-2pt}{\begin{tikzpicture} \fill (0,0) circle (1.2pt) (0,6pt) circle (1.2pt) (-5pt,10pt) circle (1.2pt)  (5pt,10pt) circle (1.2pt); \draw (0,0)--(0,6pt)--(-5pt,10pt) (0,6pt)--(5pt,10pt);\end{tikzpicture}}}
\newcommand{\Ypposet}{\raisebox{-2pt}{\begin{tikzpicture} \fill (-5pt,0) circle (1.2pt) (5pt,0pt) circle (1.2pt) (0pt,4pt) circle (1.2pt)  (0pt,10pt) circle (1.2pt); \draw (-5pt,0)--(0,4pt)--(0pt,10pt) (5pt,0pt)--(0pt,4pt);\end{tikzpicture}}}
\newcommand{\bfly}{\raisebox{-2pt}{\begin{tikzpicture} \fill (-5pt,0) circle (1.2pt) (5pt,0pt) circle (1.2pt) (-5pt,6pt) circle (1.2pt)  (5pt,6pt) circle (1.2pt); \draw (-5pt,0)--(-5pt,6pt)--(5pt,0pt)--(5pt,6pt)--(-5pt,0pt);\end{tikzpicture}}}

\begin{document}

\title{Avoiding Brooms, Forks, and Butterflies in the Linear Lattices}

\author[Shahriari]{Shahriar Shahriari}
\address{Department of Mathematics \\
Pomona College \\
Claremont, CA 91711}
\email{sshahriari@pomona.edu}

\author[Yu]{Song Yu}
\thanks{The authors thank Pomona College's Summer Undergraduate Research Program for supporting the second author.}
\address{Department of Mathematics \\
Columbia University \\
New York, NY 10027}
\email{syu@math.columbia.edu}



\date{\today}

\begin{abstract}
Let $n$ be a positive integer, $q$ a power of a prime, and $\Lnq{n}{q}$ the poset of subspaces of an $n$-dimensional vector space over a field with $q$ elements. This poset is a normalized matching poset and the set of subspaces of dimension $\floor{n/2}$ or those of dimension $\ceil{n/2}$ are the only maximum-sized anti-chains in this poset. Strengthening this well-known and celebrated result, we show that, except in the case of $\Lnq{3}{2}$, these same collections of subspaces are the only maximum-sized families in $\Lnq{n}{q}$ that avoid both a  $\wedge$ and a  $\vee$ as a subposet. We generalize some of the results to brooms and forks, and we also show that the union of the set of subspaces of dimension $k$ and $k+1$, for $k = \floor{n/2}$ or $k = \ceil{n/2}-1$, are the only maximum-sized families in $\Lnq{n}{q}$ that avoid a butterfly \bfly\ (definitions below).
\end{abstract}

\maketitle

%%%%%%%%%%%%%%%%%%%%%%%%%%%%%%%%%%%
\section{Introduction}
%%%%%%%%%%%%%%%%%%%%%%%%%%%%%%%%%%%%

Let $u$ be a positive integer, and let $Q = \{a_0, a_1, \ldots, a_u\}$ be a set of $u+1$ elements. Define a partial order on $Q$ by requiring that, for $1 \leq i \leq u$, $a_i  < a_0$ (and, for $0 \leq i \leq u$, $a_i \leq a_i$).  The poset $Q$ is called a \emph{$u$-broom}, and $a_0$ is the \emph{handle} of the broom. Likewise, $Q$ is called a \emph{$u$-fork} with $a_0$ as its \emph{handle} if, for $1 \leq i \leq u$, $a_0 < a_i$. We denote a $2$-broom by $\wedge$ and a $2$-fork  by $\vee$. See Figure \ref{fig:Broomfork}.

\begin{figure}[h]
\begin{tikzpicture}
\foreach \x in {-.5,0,0.5,1.5}{
\fill (\x,0) circle (3pt);
\draw[line width=1] (0.5,1) edge (\x,0);
}
\foreach \x in {.8,1,1.2}{
\fill (\x,0) circle (1pt);
}
\fill (0.5,1) circle (3pt);
\end{tikzpicture}\qquad
\begin{tikzpicture}[xshift=5]
\foreach \x in {-.5,0,0.5,1.5}{
\fill (\x,1) circle (3pt);
\draw[line width=1] (0.5,0) edge (\x,1);
}
\foreach \x in {.8,1,1.2}{
\fill (\x,1) circle (1pt);
}

\fill (0.5,0) circle (3pt);
\end{tikzpicture}
\caption{The Hasse diagram of a broom (left) and a fork (right)\label{fig:Broomfork}}
\end{figure}
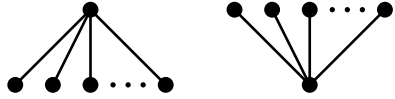

If $P$ is a poset, then $\leq_P$ denotes the partial order on $P$.
We say the poset $P$ \emph{contains} the poset $Q$ as a subposet, if there exists a subset $P_0$ of $P$ and a bijective map $\theta \colon Q \to P_0$ such that, for all $x, y \in Q$, if $x \leq_Q y$, then $\theta(x) \leq_P \theta(y)$. Note that the elements of $P_0$ are allowed to have other relations as well. Hence, for example, a chain of size $4$ contains a $3$-broom as well as a $3$-fork. If $P$ does \emph{not} contain $Q$ as a subposet, then we say that $P$ is $Q$-\emph{free}.

Let $P$, $P_1$, $\ldots$, $P_k$ be posets. Then $P$ is $(P_1, P_2, \ldots, P_k)$-\emph{free} if $P$ does not contain any of  $P_1$, $\ldots$, $P_k$ as subposets. We denote by $\mathrm{ex}(P; P_1, \ldots, P_k)$ the size of the largest subset of $P$ that is $(P_1, P_2, \ldots, P_k)$-free. (Many authors have used the notation $\mathrm{La}(P; P_1, \ldots, P_k)$ instead of $\mathrm{ex}(P; P_1, \ldots, P_k)$.)

Let $n$ be a positive integer, $q$ a power of prime, $\F_q$ a field of order $q$, and $(\F_q)^n$ a vector space of dimension $n$ over $\F_q$. We denote by $\Lnq{n}{q}$ the poset of subspaces of $(\F_q)^n$ ordered by inclusion. In analogy with the binomial coefficients, define $[n]_q = \frac{q^n-1}{q-1} = q^{n-1} + \cdots + q + 1$, $[n]_q! = [n]_q [n-1]_q \cdots [2]_q [1]_q$, and $\gbin{n}{k}_q = \frac{[n]_q!}{[k]_q!\ [n-k]_q!}$ for $0 \leq k \leq n$. The integers $\gbin{n}{k}_q$ are called \emph{Gaussian coefficients}, and, for $0 \leq k \leq n$, the number of subspaces of $(\F_q)^n$ of dimension $k$ is $\gbin{n}{k}_q$. (See van Lint and Wilson \cite[Chapter 24]{vanLintWil:01}.)
 Our first theorem is

\begin{thmA}\label{thm:Anowedgevee}
Let $n \geq 2$ be an integer, $q$ a power of a prime.  Then
$$\mathrm{ex}(\Lnq{n}{q}; \wedge, \vee) = \gbin{n}{\floor{n/2}}_q.$$
Moreover, if $\Ff$ is a subposet of $\Lnq{n}{q}$ of size $\gbin{n}{\floor{n/2}}_q$ that contains neither a $\wedge$ nor a $\vee$, then
\begin{enumerate}
\item\ if $n$ is even, then $\Ff$ consists of all subspaces of dimension $n/2$, and
\item\ if $n$ is odd, and either $n > 3$ or $q > 2$, then $\Ff$ consists either of all subspaces of dimension $\floor{n/2}$ or of all subspaces of dimension $\ceil{n/2}$, and
\item\ if $n = 3$ and $q = 2$, then there are $4$ possible configurations for $\Ff$.
\end{enumerate}
\end{thmA}

In the case $\Lnq{3}{2}$, if we ignore the trivial subspaces, and call the 1-dimensional and 2-dimensional subspaces points and lines respectively, then the inclusion relation among the points and lines is given by the familiar Fano plane. This is the one case where there are maximum-sized $(\wedge, \vee)$-free families that contain subspaces of multiple dimensions. These families are illustrated in Figure \ref{fig:FanoExtremes}. Having to account for this special case, the proof becomes a bit more subtle.

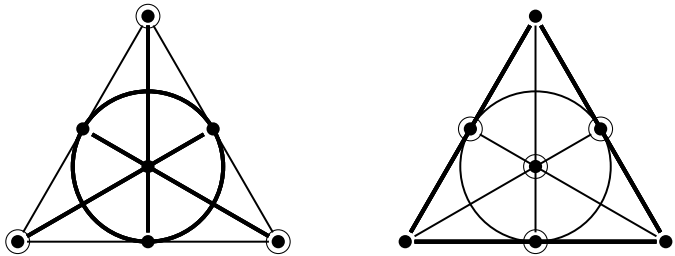
\begin{figure}
\begin{tikzpicture}
\node (P1) at (-1.732,0) {};
\node (P2) at (0,0) {};
\node (P3) at (1.732,0) {};
\node (P4) at (.866,1.5) {};
\node (P5) at (0,3) {};
\node (P6) at (-.866,1.5) {};
\node (P7) at (0,1) {};
\foreach \x in {1,2,3,4,5,6,7}{
\fill (P\x) circle (2.5pt);
}
\path[line width=.75pt] (P1) edge (P3)
(P1) edge (P4)
(P1) edge (P5)
(P3) edge (P5)
(P3) edge (P6)
(P5) edge (P2);
\draw[line width=.75pt] (P7) circle (1); 
\foreach \x in {1,3,5}{
\draw (P\x) circle (4.5pt);
\path[line width=1.5pt] (P1) edge (P4)
(P3) edge (P6)
(P5) edge (P2);
\draw[line width=1.5pt] (P7) circle (1); 
}
\end{tikzpicture}\qquad \qquad
\begin{tikzpicture}
\node (P1) at (-1.732,0) {};
\node (P2) at (0,0) {};
\node (P3) at (1.732,0) {};
\node (P4) at (.866,1.5) {};
\node (P5) at (0,3) {};
\node (P6) at (-.866,1.5) {};
\node (P7) at (0,1) {};
\foreach \x in {1,2,3,4,5,6,7}{
\fill (P\x) circle (2.5pt);
}
\path[line width=.75pt] (P1) edge (P3)
(P1) edge (P4)
(P1) edge (P5)
(P3) edge (P5)
(P3) edge (P6)
(P5) edge (P2);
\draw[line width=.75pt] (P7) circle (1); 
\foreach \x in {2,4,6,7}{
\draw (P\x) circle (4.5pt);
\path[line width=1.5pt] (P1) edge (P3)
(P1) edge (P5)
(P3) edge (P5);
}
\end{tikzpicture} 
\caption{In the Fano plane, take the three vertices of a triangle, the three edges that each contain just one of these vertices, and the one edge that contains none of these vertices. This collection of 4 edges and 3 vertices, as well as its dual on the right (the three edges of a triangle together with the vertices that are not their points of intersection), give maximum-sized $(\wedge, \vee)$-free families containing subspaces of multiple dimensions.\label{fig:FanoExtremes}}
\end{figure}

If $n$ is even, there is only one largest level in $\Lnq{n}{q}$, and a (very similar) proof for this part of Theorem A was already given by Salerno and Shahriari \cite{Salerno:09}. For this case, we actually generalize the result to brooms and forks.

\begin{thmB}\label{thm:nevenbroomfork}
Let $n$ be an even positive integer, $q$ a power of a prime, and $u$ and $v$ positive integers. Assume $u \leq q$ and $v \leq q$. Then the family of subspaces of dimension $n/2$ is the only maximum-sized $u$-broom and $v$-fork free subposet of $\Lnq{n}{q}$. In particular, the maximum size of a $u$-broom and $v$-fork free family in $\Lnq{n}{q}$ is $\gbin{n}{n/2}_q$.
\end{thmB}

We then turn to the butterfy (\bfly), the \textrm{Y} poset (\Yposet), and the \textrm{Y'} poset (\Ypposet). A \emph{butterfly} is a poset $\{a,b,c,d\}$ whose partial order is defined by $a \leq c$, $a \leq d$, $b \leq c$, and $b \leq d$. A \textrm{Y} is a butterfly with the additional requirement that $a \leq b$; dually, a \textrm{Y'} is a butterfly with the additional requirement that $c \leq d$. See Figure \ref{fig:Butterfly}.

\begin{figure}[h]
\begin{tikzpicture}
\fill (-0.5, 0) circle (3pt);
\fill (-0.5, 1) circle (3pt);
\fill (0.5, 0) circle (3pt);
\fill (0.5, 1) circle (3pt);

\draw[line width=1] (-0.5,0) edge (-0.5,1);
\draw[line width=1] (-0.5,0) edge (0.5,1);
\draw[line width=1] (0.5,0) edge (-0.5,1);
\draw[line width=1] (0.5,0) edge (0.5,1);
\end{tikzpicture}\qquad
\begin{tikzpicture}
\fill (-0.5, 0.7) circle (3pt);
\fill (0.5, 0.7) circle (3pt);
\fill (0, 0) circle (3pt);
\fill (0, -0.7) circle (3pt);

\draw[line width=1] (0,0) edge (-0.5,0.7);
\draw[line width=1] (0,0) edge (0.5,0.7);
\draw[line width=1] (0,-0.7) edge (0,0);
\end{tikzpicture}\qquad
\begin{tikzpicture}
\fill (-0.5, -0.7) circle (3pt);
\fill (0.5, -0.7) circle (3pt);
\fill (0, 0) circle (3pt);
\fill (0, 0.7) circle (3pt);

\draw[line width=1] (0,0) edge (-0.5,-0.7);
\draw[line width=1] (0,0) edge (0.5,-0.7);
\draw[line width=1] (0,0.7) edge (0,0);
\end{tikzpicture}\qquad
\caption{The Hasse diagram of a butterfly, a \textrm{Y}, and a \textrm{Y'}, respectively\label{fig:Butterfly}}
\end{figure}
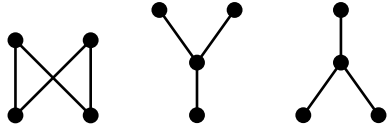

Let $\bc{\Lnq{n}{q}}{k}$ be the collection of all subspaces of dimension $k$ in $\Lnq{n}{q}$. Our result on avoiding butterflies, Ys, and Y's is

\begin{thmC}\label{thm:nobutterfly}
Let $n \geq 3$ be an integer, $q$ a power of a prime.  Then
$$\mathrm{ex}(\Lnq{n}{q}; \bfly) = \mathrm{ex}(\Lnq{n}{q}; \Yposet, \Ypposet) = \gbin{n}{\floor{n/2}}_q + \gbin{n}{\floor{n/2}+1}_q.$$
Moreover, if $\Ff$ is a subposet of $\Lnq{n}{q}$ of size $\gbin{n}{\floor{n/2}}_q+ \gbin{n}{\floor{n/2}+1}_q$ that is butterfly-free or   (Y, Y')-free, then
\begin{enumerate}
\item\ if $n$ is odd, then $\Ff = \bc{\Lnq{n}{q}}{\floor{n/2}} \cup \bc{\Lnq{n}{q}}{\ceil{n/2}}$, and
\item\ if $n$ is even, then either $\Ff = \bc{\Lnq{n}{q}}{n/2} \cup \bc{\Lnq{n}{q}}{n/2+1}$ or $\Ff = \bc{\Lnq{n}{q}}{n/2} \cup \bc{\Lnq{n}{q}}{n/2-1}$.
\end{enumerate}
\end{thmC}

In studying $\Lnq{n}{q}$, our guide is the boolean lattice  $\bool{n}$. Let $n$ be a positive integer and $[n] = \{1, 2, \ldots, n\}$ be a set with $n$ elements, then the \emph{boolean lattice} (or \emph{subset lattice}) of order $n$, denoted by $\bool{n}$, is the poset of subsets of $[n]$ ordered by inclusion. Often results about the boolean lattices have counter parts for the linear lattices, and there are many instances where if you take a result for $\Lnq{n}{q}$ and ``let $q \to 1$'' then you get the correct result for the $\bool{n}$. The literature on forbidden posets in $\bool{n}$ is vast and we limit ourselves to directly relevant items.  

Denote by $\twochain$ a chain of length $1$ (and size $2$). A collection of the elements of a poset that avoids $\twochain$ is called an \emph{anti-chain}. The celebrated Sperner Theorem (Sperner \cite{Sperner:28}) says that $\mathrm{ex}(\bool{n}; \twochain\ ) = \binom{n}{\floor{n/2}}$ and that the only $\twochain$-free families of size $\binom{n}{\floor{n/2}}$ in $\bool{n}$ are the collection of subsets of size $\floor{n/2}$ and the collection of subsets of size $\ceil{n/2}$. The corresponding ``strict Sperner property'' for $\Lnq{n}{q}$ is identical and well known: 
\begin{thm}[See Example 4.6.2, page 175 of Engel \cite{Engel:97}]\label{thm:maxantichainLnq} The largest size of an anti-chain in $\Lnq{n}{q}$ is $\gbin{n}{\floor{n/2}}_q$ and the only anti-chains of this size are the collection of subspaces of dimension $\floor{n/2}$ and the collection of subspaces of dimension $\ceil{n/2}$. 
\end{thm} 
To avoid $\twochain$, you certainly have to avoid both $\wedge$ and $\vee$. Hence maximum-sized anti-chains in $\Lnq{n}{q}$ are among the configurations outlined in Theorem A. Hence, the Sperner property of $\Lnq{n}{q}$ (including the characterization of the maximum-sized anti-chains)  is a corollary of our Theorem A. While in the linear lattices, except for $\Lnq{3}{2}$, the configurations that give a maximum-sized $(\wedge, \vee)$-free family are identical to the maximum-sized anti-chains, this is not so in the boolean lattices.
\begin{thm}[Katona, Tarj\'an, 1983 \cite{KatonaTarjan:83}]
For all $n \geq 2$,
$$\mathrm{ex}\left(\bool{n}; \wedge, \vee \right)=\ 2 \bc{n-1}{\floor{\frac{n-1}{2}}}.$$
\end{thm}
Note that if $n$ is even, $2 \bc{n-1}{\floor{\frac{n-1}{2}}} = \bc{n}{\floor{n/2}}$, but if $n$ is odd, $2 \bc{n-1}{\floor{\frac{n-1}{2}}} > \bc{n}{\floor{n/2}}$. Hence for odd $n$, the situation in the linear lattices is qualitatively different than that of the boolean lattices.

For butterflies, our Theorem C mirrors its Boolean lattice counterpart. We denote all subsets of size $k$ in $[n]$ by $\bc{[n]}{k}$.

\begin{thm}[De Bonis, Katona, Swanepoel,  2005 \cite{DeBonisKatonaSwa:05}] \label{thm:nobutterflyboolean}
For all $n \geq 3$,
$$\mathrm{ex}\left(\bool{n}; \bfly \right)= \bc{n}{\floor{\frac{n}{2}}} + \bc{n}{\floor{\frac{n}{2}}+1}.$$
Moreover, if $\Ff$ is a subposet of $\bool{n}$ of size $\bc{n}{\floor{\frac{n}{2}}} + \bc{n}{\floor{\frac{n}{2}}+1}$ that is butterfly-free, then
\begin{enumerate}
\item\ if $n$ is odd, then $\Ff = \bc{[n]}{\floor{n/2}} \cup \bc{[n]}{\ceil{n/2}}$, and
\item\ if $n$ is even, and $n \geq 6$, then either $\Ff = \bc{[n]}{n/2} \cup \bc{[n]}{n/2+1}$  or $\Ff = \bc{[n]}{n/2} \cup \bc{[n]}{n/2-1}$, and
\item\ if $n=4$, then either $\Ff = \bc{[4]}{2} \cup \bc{[4]}{3}$ or $\Ff = \bc{[4]}{2} \cup \bc{[4]}{1}$, or $\Ff$ is isomorphic to the following family:
$$\bc{[4]}{2} \cup \{\{1\}, \{2,3,4\}, \{2\}, \{1,3,4\}\}.$$
\end{enumerate}
\end{thm}

For the boolean lattices, many other forbidden configurations have been studied. We refer the reader to Katona \cite{Katona:08}, Griggs et al. \cite{GriggsLiLu:12}, Gr\'osz et al. \cite{GroszMetTomp:18}, and Nagy \cite{Nagy:18} for references and surveys.  For the linear lattices, other than the present study, Sarkis et al. \cite{PCURC-S11:14} have considered ``diamond''-free collections of subspaces in $\Lnq{n}{q}$.

%%%%%%%%%%%%%%%%%
\section{Preliminaries on Normalized Matching Posets and Linear Lattices}
%%%%%%%%%%%%%%%%%%
A totally ordered subset of a poset is called a \emph{chain}, and a poset $P$ is \emph{graded} if all maximal chains are of equal length. For a graded poset $P$, the \emph{rank} of an element $x \in P$ is the length of a maximal chain from a minimal element of $P$ to $x$, and the rank of $P$ is the length of a maximal chain in $P$. In a graded poset $P$, for a non-negative integer $i$, the set of all elements of rank $i$ is called the $i$th \emph{level} of $P$ and denoted by $\bc{P}{i}$.  The boolean lattice $\bool{n}$ and the linear lattice $\Lnq{n}{q}$ are both graded posets of rank $n$. In $\bool{n}$ and $\Lnq{n}{q}$, the rank of an element $x$ is the size of $x$ and the dimension of $x$ respectively. Hence, for $x \in \Lnq{n}{q}$, to say $\dim(x) = i$, $x$ is of rank $i$ in $\Lnq{n}{q}$, or $x$ is on the $i$th level of $\Lnq{n}{q}$ are equivalent. 

For a graded poset $P$, the sizes of the different levels of $P$ are called the \emph{rank numbers} of $P$. The rank numbers of $\Lnq{n}{q}$ are given by the Gaussian coefficients $\gbin{n}{0}_q$,  $\gbin{n}{1}_q$, $\ldots$,  $\gbin{n}{n}_q$. This sequence is symmetric (that is, $\gbin{n}{k}_q = \gbin{n}{n-k}_q$) and unimodal with $\gbin{n}{\floor{n/2}}_q = \gbin{n}{\ceil{n/2}}_q$ being the largest rank number(s).

If $P$ is a poset, then its \emph{Hasse diagram} is a graph whose vertices are elements of $P$, and $x$ is adjacent to $y$ if $x < y$ and there exists no $z \in P$ with $x < z < y$. If $P$ is a graded poset of rank $n$, $0< i \leq n$, and $\Aa \subseteq \bc{P}{i}$, then the \emph{shadow} of $\Aa$, denoted by $\shadow \Aa$, is the set of elements of $\bc{P}{i-1}$ that are related to some element of $\Aa$.  Likewise, if $0 \leq i < n$, and $\Aa \subseteq \bc{P}{i}$, then the \emph{shade} of $\Aa$, denoted by $\shade \Aa$, is the set of elements of $\bc{P}{i+1}$ that are related to some element of $\Aa$. If $\Aa= \{a\}$ is a singleton set, we will also use the notation $\shadow a$ and $\shade a$ for $\shadow \{a\}$ and $\shade \{a\}$. Thinking of the Hasse diagram as a graph, we sometimes use graph-theoretic terminology and refer to $\shadow \Aa$ and $\shade \Aa$ as the \emph{neighbors} of $\Aa$ in $\bc{P}{i-1}$ and $\bc{P}{i+1}$ respectively.

If we fix a basis for $(\F_q)^n$, for every subspace $W$ of $(\F_q)^n$ of dimension $k$, there is a unique $k \times n$ matrix $M_W$ of rank $k$ and in reduced row echelon form such that $W$ is the row space of $M_W$.  Define a map $\theta \colon W \mapsto\ \mathrm{nullspace}(M_W)$, then $\theta$ is an order reversing bijection of $\Lnq{n}{q}$.  As a result, $\Lnq{n}{q}$ is ``symmetric'' around the middle rank, and turning the Hasse diagram of the poset ``upside-down'' results in the same Hasse diagram.

A class of posets that includes both the boolean lattices and the linear lattices (as well as the poset of positive integer divisors of a positive integer ordered by divisibility) is that of normalized matching posets.

\begin{defn}[Graham and Harper \cite{GrahamHar:69}]
Let $P$ be a graded poset of rank $n$. Assume that for every integer $i$  with $0< i \leq n$, and for every $\Aa \subseteq \bc{P}{i}$, we have  $$\frac{\size{\shadow \Aa}}{\size{\bc{P}{i-1}}} \geq \frac{\size{\Aa}}{\size{\bc{P}{i}}}.$$
We then say that $P$ is a \emph{normalized matching} poset.
\end{defn}

It is straightforward to show that in a normalized matching poset, for every $0\leq i < n$, and for every $\Aa \subseteq \bc{P}{i}$, we also have  $\frac{\size{\shade \Aa}}{\size{\bc{P}{i+1}}} \geq \frac{\size{\Aa}}{\size{\bc{P}{i}}}$. Having the normalized matching property is also equivalent to having the LYM property (Kleitman \cite{Kleitman:74}). A graded poset $P$ of rank $n$ has the \emph{LYM property} if for all anti-chains $\Aa \subseteq P$, we have $$\sum_{i = 0}^n \dfrac{\size{\Aa \cap \bc{P}{i}}}{\size{\bc{P}{i}}} \leq 1.$$ 

\begin{defn}
Let $P$ be a graded poset of rank $n$. Assume that for every $0 < i \leq n$, and for every $x, y \in \binom{P}{i}$, we have $\size{\shadow x} =  \size{\shadow y}$, and for every $0 \leq j < n$, and for every $z, w \in \binom{P}{j}$, we have $\size{\shade z} =  \size{\shade w}$. We say that $P$ is a \emph{regular} poset.
\end{defn}

\begin{lem}[Baker \cite{Baker:69}] Every regular poset is normalized matching. In particular, for $n$ a positive integer, and $q$ a prime power, the poset $\Lnq{n}{q}$ (as well as $\bool{n}$) is a normalized matching poset.
\end{lem}

%%%%%%%%%%%%%%%%
\section{$u$-brooms and $v$-forks}
%%%%%%%%%%%%%%%%%%%

 We first show that  a maximum-sized family of subspaces that avoids a $u$-broom and a $v$-fork can be ``pushed'' into the middle levels of the linear lattice.  If $P$ is a graded poset and $Q$ is a subposet of $P$, then we consider replacing the elements of highest rank in $Q$ with their shadow in $P$.  
 We start with a straightforward observation.

 \begin{lem}\label{lem:ToShadowBroomFree}
 Let $u$, $i$, and $n$ be positive integers with $i \leq n$, and let $P$ be a graded poset of rank $n$.  Let $Q$ be a subposet of $P$ with no elements of rank greater than $i$, and let $\Bb = \bc{P}{i} \cap Q$ denote the elements of rank $i$ in $Q$. Let $Q^\prime = \left(Q - \Bb \right)  \cup \shadow\Bb$ be the poset obtained by replacing elements of $\Bb$ with their shadow in $P$. Then $Q^\prime$ has no element of rank greater than $i-1$, and if $Q$ is $u$-broom free then so is $Q^\prime$.
 \end{lem}

\begin{proof}
It is clear that $Q^\prime$ has no element of rank greater than $i-1$. Now assume that $Q$ is $u$-broom free but $Q^\prime$ contains a $u$-broom. Since $Q$ did not contain a $u$-broom, the handle of a newly created $u$-broom in $Q^\prime$ would have to be in $\shadow \Bb$, and the rest of the elements of this $u$-broom would be in $Q$. But if the handle is in $\shadow \Bb$ then it is comparable to an element of $\Bb$. Replacing the handle with this element of $\Bb$ gives a $u$-broom in $Q$. The contradiction completes the proof. 
\end{proof}

Specializing to normalized matching posets, we can quantify our observation.

\begin{prop}\label{prop:NMbroomfree}
 Let $u$, $i$, and $n$ be positive integers with $i \leq n$, and let $P$ be a normalized matching poset of rank $n$ with rank numbers $r_0$, $r_1$, $\ldots$, $r_n$.  Let $Q$ be a subposet of $P$ containing no elements of $P$ of rank greater than $i$. Let $\Bb = \bc{P}{i} \cap Q$  and $Q^\prime =  \left(Q - \Bb \right)  \cup \shadow\Bb$. Assume that $Q$ is $u$-broom free and that  $\frac{r_{i-1}}{r_i} = u + \alpha$ for some non-negative real number $\alpha$, then
 $$\size{Q^\prime} \geq \size{Q} + \alpha\size{\Bb}.$$
In particular, $\size{Q^\prime} \geq \size{Q}$, and if $\alpha > 0$, then $\size{Q^\prime} > \size{Q}$.
\end{prop}

\begin{proof}
Since $Q$ does not contain a $u$-broom, an element of $\Bb$ can be comparable to at most $u-1$ elements of $Q$.  Hence, $\size{\shadow \Bb \cap Q} \leq (u-1)\size{\Bb}$. In addition, since $P$ is a normalized matching poset, we have
$$ \size{\shadow \Bb} \geq \frac{r_{i-1}}{r_i} \size{\Bb} = (u+\alpha) \size{\Bb}.$$
As a result, 
$$\size{\shadow \Bb - Q} = \size{\shadow \Bb} - \size{\shadow \Bb \cap Q} \geq (u+\alpha) \size{\Bb} -(u-1)\size{\Bb} = (\alpha+1)\size{\Bb}.$$
We conclude that
$$\size{Q^\prime} = \size{Q} - \size{\Bb} + \size{\shadow \Bb - Q} \geq \size{Q} + \alpha\size{\Bb}.$$
\end{proof}

In the situation of Proposition \ref{prop:NMbroomfree}, let $a \leq n$ be a positive integer such that $\frac{r_{i-1}}{r_i} > u$ for all $a \leq i \leq n$. If $Q$ is an arbitrary subposet of $P$ that does not contain a $u$-broom, then by repeated replacement of the highest ranked elements with their shadow, we get a larger $u$-broom free subposet whose elements have rank at most $a-1$. We can thus conclude that the elements of a maximum-sized $u$-broom free subposet are restricted to ranks $0$ through $a-1$. However, if $Q$ is both $u$-broom free and $v$-fork free, then replacing the highest rank elements of $Q$ with their shadow (as done in Lemma \ref{lem:ToShadowBroomFree}) may create a $v$-fork. To also avoid $v$-forks, we need to match and replace the highest ranked elements of $Q$ with an equal number of elements of $\shadow \Bb - Q$. This will result in a poset that continues not to contain $u$-brooms and $v$-forks, and will have the same size as our original poset.

\begin{prop}\label{prop:Pushdown}
 Let $u$, $v$, $i$, and $n$ be positive integers with $i \leq n$, and $P$ be a normalized matching poset of rank $n$ with rank numbers $r_0$, $r_1$, $\ldots$, $r_n$.  Let $Q$ be a subposet of $P$ containing no elements of $P$ of rank greater than $i$.  Assume that $Q$ does not contain $u$-brooms or $v$-forks and that  $\frac{r_{i-1}}{r_i} \geq u$, then we can appropriately replace all the elements of $\Bb =\bc{P}{i} \cap Q$ with an equal number of elements in $\shadow \Bb - Q$ resulting in a poset $Q^\prime$ that has as many elements as $Q$, has no elements of $P$ of rank greater than or equal to $i$, and contains no $u$-brooms or $v$-forks. 
 \end{prop}
 
\begin{proof}
We verify the marriage condition to show that we can  match elements of $\Bb$ with an equal number of elements of $\shadow \Bb - Q$, in such a way that every element $b$ of $\Bb$ is matched with a distinct element of $\shadow \Bb-Q$ that is covered by $b$. Replacing every element of $\Bb$ with its match in $\shadow \Bb - Q$, we get the desired subposet $Q^\prime$ of $P$.  By Lemma \ref{lem:ToShadowBroomFree}, $Q^\prime$ will not contain a $u$-broom. If $Q^\prime$ had a new $v$-fork, then some of the elements of the $v$-fork would have to be in $\shadow \Bb - Q$. Replacing these elements with their matches in $\Bb$ gives a $v$-fork in $Q$. The contradiction shows that $Q^\prime$ is also $v$-fork free.

To verify the marriage condition let $X \subseteq \Bb$, and consider the poset $(Q-\Bb) \cup X$, then by Proposition \ref{prop:NMbroomfree}, in this poset, if we replace the elements of $X$ by their shadow, we get a new poset of at least the same size. Hence $X$ has at least $\size{X}$ many neighbors in $\shadow \Bb - Q$. This verifies the marriage condition and shows that we can match the elements of $\Bb$ with those of $\shadow \Bb- Q$ as desired. The proof is now complete.
\end{proof}

With an identical proof, we can get dual versions of Lemma \ref{lem:ToShadowBroomFree} and Proposition \ref{prop:NMbroomfree} for $v$-fork free families. Let $n$, $a$, and $v$ be positive integers with $a < n$, such that $\frac{r_{i+1}}{r_i} > v$ for all $0 \leq i \leq a$. If $Q$ is an arbitrary subposet of $P$ that does not contain a $v$-fork, then by repeated replacement of the lowest ranked elements with their shade, we get a larger $v$-fork free subposet whose elements have rank at least $a+1$. We thus conclude that, under such circumnstances, the elements of a maximum-sized $v$-fork free subposet are restricted to ranks $a+1$ through $n$. Here, we record the dual version of Proposition \ref{prop:Pushdown}.

\begin{prop}\label{prop:Pushup}
 Let $u$, $v$, and $n$ be positive integers, $i$ a non-negative integer with $i < n$, and let $P$ be a normalized matching poset of rank $n$ with rank numbers $r_0$, $r_1$, $\ldots$, $r_n$.  Let $Q$ be a subposet of $P$ containing no elements of $P$ of rank less than $i$.  Assume that $Q$ does not contain $u$-brooms or $v$-forks and that  $\frac{r_{i+1}}{r_i} \geq v$, then we can appropriately replace all the elements of $\Bb = \bc{P}{i} \cap Q$ with an equal number of elements in $\shade \Bb - Q$ resulting in a poset $Q^\prime$ that has as many elements as $Q$, has no elements of $P$ of rank less than or equal to $i$, and contains no $u$-brooms or $v$-forks. 
 \end{prop}

Propositions \ref{prop:Pushdown} and \ref{prop:Pushup} have an immediate corollary for the linear lattices.

 \begin{cor}\label{cor:Lnqbroomfork1}
 Let $n$ be a positive integer, $q$ a power of a prime, and $u$ and $v$ positive integers.  Then there exists a maximum-sized $u$-broom and $v$-fork free family $\Ff \subseteq \Lnq{n}{q}$ such that for all $F \in \Ff$,
 $$\floor{\frac{n+1-\log_q v}{2}} \leq \dim F \leq \ceil{\frac{n-1+\log_q u}{2}}.$$

 \end{cor}

\begin{proof}
The poset $\Lnq{n}{q}$ is normalized matching and its rank numbers are $r_0 = \gbin{n}{0}_q$, $r_1= \gbin{n}{1}_q$, $\ldots$, $r_n = \gbin{n}{n}_q$. If $n$ is odd and $i = \frac{n+1}{2}$, then $r_{i-1} = r_i$. However, for $i > \frac{n+1}{2}$ (with $n$ even or odd), we have
$$\frac{r_{i-1}}{r_i} = \frac{\gbin{n}{i-1}_q}{\gbin{n}{i}_q} = \frac{[i]_q}{[n-i+1]_q} = \frac{q^i-1}{q^{n-i+1}-1} > q^{2i-n-1},$$
and $q^{2i-n-1} \geq u$ if $i \geq \frac{n+1+\log_q(u)}{2}$. Now if $Q \subseteq \Lnq{n}{q}$ is a maximum-sized $u$-broom and $v$-fork free family, then repeated application of  Proposition \ref{prop:Pushdown}, starting with $i = n$ and ending with $i = \ceil{\frac{n+1+\log_q(u)}{2}}$, gives an equal size $u$-broom and $v$-fork free family $Q^\prime$ confined to rank numbers $0$ through $\ceil{\frac{n+1+\log_q(u)}{2}}-1 = \ceil{\frac{n-1+\log_q(u)}{2}}$. 

Similarly, for $i < \frac{n-1}{2}$, we have
$$\frac{r_{i+1}}{r_i} = \frac{\gbin{n}{i+1}_q}{\gbin{n}{i}_q} = \frac{q^{n-i}-1}{q^{i+1}-1} > q^{n-2i-1},$$
and $q^{n - 2i - 1} \geq v$ if $i \leq \frac{n-1-\log_q(v)}{2}$. Apply Proposition \ref{prop:Pushup} repeatedly to $Q^\prime$, starting with $i = 0$ and ending with $i = \floor{\frac{n-1-\log_q(v)}{2}}$, to get $\Ff$, a maximum-sized $u$-broom and $v$-fork free subposet of $\Lnq{n}{q}$, confined to rank numbers $\floor{\frac{n-1-\log_q(v)}{2}} + 1 = \floor{\frac{n+1-\log_q(v)}{2}}$ through $\ceil{\frac{n-1+\log_q(u)}{2}}$. 
\end{proof}

We can already prove Theorem \hyperlink{thmB}{B}.

\begin{proof}[\textbf{Proof of Theorem \hyperlink{thmB}{B}}] If $u$ and $v$ are less than or equal to $q$, then the restriction on the dimension of the subspaces in the maximum-sized family not containing $u$-brooms and $v$-forks produced in Corollary \ref{cor:Lnqbroomfork1} becomes
$$\floor{\frac{n}{2}} \leq \dim(F) \leq \ceil{\frac{n}{2}}.$$
However, since $n$ is assumed to be even, we get that $\dim(F) = n/2$ for all $F \in \Ff$. We conclude that the collection of all subspaces of dimension $n/2$ is a maximum-sized family not containing $u$-brooms and $v$-forks. 

To show uniqueness, let $Q$ be any $u$-broom and $v$-fork free family of subspaces in $\Lnq{n}{q}$ that contains at least one subspace of dimension other than $n/2$. Without loss of generality, assume that $Q$ contains a subspace of dimension greater than $n/2$. Use Propositions \ref{prop:Pushdown} and \ref{prop:Pushup} (as in the proof of Corollary \ref{cor:Lnqbroomfork1}) to construct a $u$-broom and $v$-fork family $Q^\prime$ with $\size{Q^\prime} = \size{Q}$ and so that $Q^\prime$ is confined to ranks $n/2$ and $n/2 + 1$ (in other words, unlike the proof of Corollary \ref{cor:Lnqbroomfork1}, stop one step short of pushing every element into the middle level). Since $\gbin{n}{n/2}_q/\gbin{n}{n/2+1}_q > q \geq u$, by  Proposition \ref{prop:NMbroomfree} (with $\alpha > 0$ and $i =n/2+1$), if we replace the elements of rank $n/2+1$ with their shadow, we get a family that is larger than $Q$ and confined to just rank $n/2$. This new bigger family will certainly be $u$-broom and $v$-fork free since it consists of elements of just one rank, and this proves that $Q$ could not have been a maximum-sized $u$-broom and $v$-fork free family. The proof is now complete.
\end{proof}

In the case, when $n$ is odd, and, as long as $u, v \leq q^2$, Corollary \ref{cor:Lnqbroomfork1} shows the existence of a maximum-sized $u$-broom and $v$-fork free family of subspaces confined to the middle two levels of $\Lnq{n}{q}$. The uniqueness proof that we gave for even $n$, however, does not directly generalize to this case.

%%%%%%%%%%%%%%%%%%%%%%%%%%%%%%%%
\section{$2$-brooms and $2$-forks}
%%%%%%%%%%%%%%%%%%%%%%%%%%%%%%%%

We now focus on families of $\Lnq{n}{q}$ that are $(\wedge, \vee)$-free.  We begin with the particular case of $n = 3$, which we will need as the base case for the later inductive proof of Theorem \hyperlink{thmA}{A}, and which poses special issues because of the examples given in Figure \ref{fig:FanoExtremes}.

\begin{prop}\label{prop:L3q}
Let $q$ be a power of a prime, and let $\Ff$ be a maximum-sized $(\wedge, \vee)$-free family in $\Lnq{3}{q}$. Then
\ben
\item\ if $q \neq 2$, then $\Ff$ consists of all subspaces of dimension $1$ or all subspaces of dimension $2$, and
\item\ if $q = 2$, then $\Ff$ consists either entirely of subspaces of dimension $1$, or entirely of subspaces of dimension $2$ ,or $\Ff$ is one of the two configurations illustrated in Figure \ref{fig:FanoExtremes}.
\een
In particular, 
$$ \mathrm{ex}(\Lnq{3}{q}; \wedge, \vee) = \gbin{3}{1}_q = q^2 + q + 1.$$
\end{prop}

\begin{proof}
The collection of elements of rank $1$ (or rank $2$) in $\Lnq{3}{q}$ certainly forms a $(\wedge, \vee)$-free subposet of size $\gbin{3}{1}_q = q^2+q+1$. Hence, $\size{\Ff} \geq q^2+q+1 \geq 7$. An element of $\mathcal{F}$ can be comparable to at most one other element of $\mathcal{F}$ since otherwise $\mathcal{F}$ will contain a $\wedge$ or a $\vee$. Thus $\Ff$ cannot contain the maximal or the minimal element of $\Lnq{3}{q}$ (the elements of rank 3 and rank 0) since otherwise it could have at most two elements. Let $\Aa$ and $\Bb$ consist, respectively, of elements of rank $1$ and $2$ of  $\mathcal{F}$ that are \emph{not} comparable to any other element of $\mathcal{F}$. Likewise $\Xx$ and $\Yy$ are, respectively, elements of rank $1$ and $2$ of $\mathcal{F}$ that are comparable to exactly one other element of $\Ff$. The families $\Aa$, $\Bb$, $\Xx$, and $\Yy$ give a partition of $\mathcal{F}$, and each element of $\Xx$ is comparable to exactly one element of $\Yy$. As a result $\size{\Xx} = \size{\Yy}$.  Let $\Dd$ and $\Ee$ consist, respectively, of elements of rank $1$ and $2$ of $\Lnq{3}{q}$ that are not in $\mathcal{F}$. See Figure \ref{fig:L3qwedgeveefree}.
 
 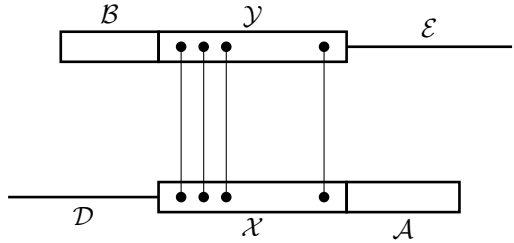
\begin{figure}[h]
 \begin{tikzpicture}
 \draw[line width=1pt] (0,0)--node[below]{$\Dd$}(2,0);
 \draw[line width = 1pt] (2,.2) rectangle node[below=4pt]{$\Xx$} (4.5,-.2);
 \draw[line width = 1pt] (4.5,.2) rectangle node[below=5pt]{$\Aa$}  (6,-.2);
  \draw[line width=1pt] (4.5,2)--node[above]{$\Ee$}(6.7,2);
 \draw[line width = 1pt] (2,2.2) rectangle node[above=4pt]{$\Yy$}  (4.5,1.8);
  \draw[line width = 1pt] (.7,2.2) rectangle node[above=5pt]{$\Bb$}  (2,1.8) ;
 \foreach \x in {2.3,2.6,2.9,4.2}{
 \fill (\x,2) circle (2pt);
  \fill (\x,0) circle (2pt);
  \draw (\x,0)--(\x,2);
 }
 \end{tikzpicture}
 \caption{The families $\Aa$, $\Bb$, $\Xx$, and $\Yy$ give a partition of the $(\wedge, \vee)$-free family $\Ff$.\label{fig:L3qwedgeveefree}}
 \end{figure}
 
Without loss of generality, and by symmetry, we assume that $\size{\Bb} \leq \size{\Aa}$, and since $\size{\Ff} \geq q^2 + q + 1$, we let $\size{\Ff} = q^2 + q + 1 + t$ for some non-negative integer $t$. 

\newcounter{mycl}
\setcounter{mycl}{0}
\newcommand{\myclaim}{\refstepcounter{mycl}\noindent\textsc{Claim \themycl:\/}}

\bigskip

\myclaim\ \label{cl:Ddnotempty} If  $\Dd$ is empty, then $t = 0$ and $\Ff = \Aa$ consists of all subspaces of dimension $1$.
\begin{proof}[Proof of Claim]
Elements of $\Bb$ have all their neighbors in $\Dd$ and elements of $\Yy$ have $q$ of their neighbors in $\Dd$. So if $\size{\Dd} = 0$, then $\Bb$ and $\Yy$ are both empty. The set $\Xx$ has as many elements as $\Yy$ and so it must be empty as well. Thus $\Ff = \Aa$ must consists of all subspaces of dimension $1$. 
\end{proof}

Because of Claim \ref{cl:Ddnotempty}, from here on, we assume that $\size{\Dd} > 0$, and we aim to show that this is only possible if $q=2$ and we have the configuration on the right of Figure \ref{fig:FanoExtremes}.

\bigskip

\myclaim\ \label{cl:DneighborinE} If $\Aa$ is non-empty, then every element of $\Dd$ has at least one neighbor in $\Ee$.
\begin{proof}[Proof of Claim]
Let $w$ be an arbitrary 1-dimensional subspace in $\Lnq{3}{q}$. Note that $\size{\shade w} = q+1$, and the intersection of each pair of 2-dimensional subspaces in $\shade w$ is exactly $w$. (If $u$ and $v$ are distinct 2-dimensional subspaces then $\dim(u \cap v) = \dim(u)+\dim(v)-\dim(u + v) = 1$.) Each of these 2-dimensional subspaces has $q+1$ 1-dimensional subspaces, and so $\size{\shadow(\shade w)} = (q+1)q + 1 = q^2 + q + 1 = \gbin{3}{1}$. We conclude that $\shadow(\shade w)$ is the entire collection of 1-dimensional subspaces in $\Lnq{3}{q}$.
Let $d \in \Dd$. We just proved that $\shadow(\shade d)$ contains all of $\Aa$, but the elements of $\Bb \cup \Yy$ have no neighbors in $\Aa$. We conclude that $\shade d$ must have some elements from $\Ee$. 
\end{proof}

\myclaim\ \label{cl:sizeDpt}$\size{\Dd} + t = \size{\Bb} + \size{\Yy}$, and $\size{\Dd} \leq \frac{q^2+q}{2}$.
\begin{proof}[Proof of Claim]  From the definitions, we have $\size{\Aa} + \size{\Xx} + \size{\Bb} + \size{\Yy} = \size{\Ff} = q^2+q+1+t $ and $\size{\Dd} + \size{\Xx} + \size{\Aa} =\gbin{3}{1}_q = q^2+q+1$. Combining them we get $\size{\Dd} + t = \size{\Bb} + \size{\Yy}$. Also, since $\size{\Aa} \geq \size{\Bb}$ and $\size{\Xx} = \size{\Yy}$ we have $2(\size{\Aa} + \size{\Xx}) \geq q^2 + q + 1$ and hence, since $q^2+q+1$ is odd and $\size{\Aa} + \size{\Xx}$ is an integer,  $\size{\Aa} + \size{\Xx} \geq \frac{q^2+q}{2} + 1$. It follows that $\size{\Dd} = q^2+q+1 - \size{\Aa} - \size{\Xx} \leq \frac{q^2+q}{2}$.
\end{proof}

\myclaim\ \label{cl:sizeYyq} $\size{\Yy} \geq q$.

\begin{proof}[Proof of Claim] If $\Aa$ was empty, then $\Bb$ is also empty (recall that we are assuming $\size{\Aa} \geq \size{\Bb}$) and $\Ff = \Xx \cup \Yy$. So $\size{\Yy} = \size{\Ff}/2 \geq q$. So assume $\Aa$ is not empty and denote by $E(\Bb \cup \Yy, \Dd)$ the edges between $\Bb \cup \Yy$ and $\Dd$. By Claim \ref{cl:DneighborinE}, every $d \in \Dd$ must have at least one neighbor in $\Ee$ and so has at most $q$ neighbors in $\Bb$. On the other hand, all the neighbors of elements of $\Bb$ are in $\Dd$, and  every element of $\Yy$ has one neighbor in $\Xx$ and $q$ neighbors in $\Dd$. Counting $\size{E(\Bb \cup \Yy, \Dd)}$ in two different ways, we have 
\begin{align*} q \size{\Dd} \geq \size{E(\Bb \cup \Yy, \Dd)} & =(q+1)\size{\Bb} + q\size{\Yy}\\
&= (q+1)\left(\size{\Bb} + \size{\Yy}\right) - \size{\Yy} \\
& = (q+1)(\size{\Dd}+t) - \size{\Yy} \qquad (\mbox{by Claim \ref{cl:sizeDpt}}).
\end{align*}
As a result, $\size{\Yy} \geq (q+1)t + \size{\Dd} \geq \size{\Dd} > 0$. Let $y \in \Yy$. The element $y$ has $q$ neighbors in $\Dd$. So $\size{\Yy} \geq \size{\Dd} \geq q$.
\end{proof}

Let $\Yy = \{y_1, y_2, \ldots, y_\ell\}$ (where $\ell$ is an integer no less than $q$), and define a sequence of subsets of $\Dd$ as follows:
\begin{align*}
T_1 & =  \shadow y_1 - \Xx \\
T_2 & =  \shadow y_2 - \left(\Xx \cup T_1\right) \\
T_3 & =  \shadow y_3 - \left(\Xx \cup T_1 \cup T_2\right) \\
\vdots \\
T_{\ell} & =  \shadow y_{\ell} - \left(\Xx \cup T_1 \cup \cdots \cup T_{\ell-1}\right)
\end{align*}

\bigskip

\myclaim\ \label{cl:uintTi} Let $i$ be an integer with $1 \leq i \leq \ell$ and $u$ an arbitrary element of rank $2$ in $\Lnq{3}{q}$ with $u \neq y_i$. Then $\shadow u$ has at most one element in $T_i$.
\begin{proof}[Proof of Claim]
Both $u$ and $y_i$ are 2-dimensional subspaces and $\dim(u \cap y_i)  = 1$. Hence, $\shadow u \cap \shadow y_i = \{u \cap y_i\}$ contains a unique element of rank $1$ in $\Lnq{3}{q}$, and the only possible element in $\shadow u \cap T_i$.
\end{proof}

\myclaim\ \label{cl:sizeDd}  $\size{\Bb} = 0$ and $\size{\Dd} = \frac{q(q+1)}{2}$. Moreover, for $1 \leq i \leq q$, $\size{T_i} = q+1-i$, and $\{T_1, T_2, \ldots, T_q\}$ partition $\Dd$.
\begin{proof}[Proof of Claim]
Let $1 \leq i \leq \ell$. Recall that $|\shadow y_i| = q+1$, and $y_i$ has exactly one neighbor in $\Xx$. For each $1 \leq j < i$, Claim \ref{cl:uintTi} implies that $\size{T_i \cap \shadow y_j} \leq 1$. Thus we have that $\size{T_i} \geq q + 1 - i$.
Since $\ell \geq q$ (by Claim \ref{cl:sizeYyq}) we have
$$\size{\Dd} \geq \size{T_1} + \cdots + \size{T_{q}} \geq q + (q-1) + \cdots +1 = \frac{q(q+1)}{2}.$$
Because of Claim \ref{cl:sizeDpt}, we now have  $\size{\Dd} = \frac{q(q+1)}{2}$. This means that every inequality in the above equation is actually an equality, and, for $1 \leq i \leq q$, $\size{T_i} = q+1-i$, and $T_1$, $\ldots$, $T_q$ partition $\Dd$. Now if $b \in \Bb$, then all $q+1$ elements of $\shadow b$ would be in $\Dd$, and so, by the pigeon hole principle, two of these elements would have to be in the same $T_i$. But this contradicts Claim \ref{cl:uintTi}, proving that $B = \emptyset$.
\end{proof}

\myclaim\ \label{cl:qis2} $q = 2$, $t = 0$, $\size{\Aa} = 1$, $\size{\Xx} = \size{\Yy} = 3$.

\begin{proof}[Proof of Claim]
Since, by Claim \ref{cl:sizeDd}, $\size{\Bb} = 0$ and $\size{\Dd} = \frac{q^2+q}{2}$, we have by Claim \ref{cl:sizeDpt} that $\ell = \size{\Yy} = \size{\Dd} + t = \frac{q^2+q}{2} + t$. Now 
$$q^2 + q + 1 + t = \size{\Ff} = \size{\Aa} + 2\size{\Yy} = \size{\Aa} + q^2+q+ 2t.$$
Thus $\size{\Aa} + t = 1$. So either $\size{\Aa} = 1$ and $t = 0$, or $\size{\Aa} = 0$ and $t = 1$.

Focus on the set $T_q \subseteq D$. By Claim \ref{cl:sizeDd}, $\size{T_q} = 1$, and so let $T_q = \{d\}$. The element $d$ has $q+1$ neighbors in $\Yy \cup \Ee$. Recall that $\Yy=\{y_1, y_2, \ldots, y_\ell\}$, each $y_i$ has $q$ neighbors in $\Dd$, and $T_1$, $\ldots$, $T_q$ partition $\Dd$. Hence, by Claim \ref{cl:uintTi}, each of $y_q$, $\ldots$, $y_\ell$ has exactly one neighbor in each of $T_1$, $\ldots$, $T_q$. In particular, $d$ neighbors with each of $y_q$, $\ldots$ $y_{\ell}$. Thus $d$ has at least $\ell-q+1$ neighbors in $\Yy$. In the case $\size{\Aa} = 0$, we have $\ell =  \frac{q^2+q}{2} +1$, and so we have to have $\frac{q^2+q}{2}+1 -q +1 \leq q+1$. In the case $\size{\Aa} = 1$, $\ell = \frac{q^2+q}{2}$, but in this case, $d$ has at least one neighbor in $\Ee$ (Claim \ref{cl:DneighborinE}). So we must have $\frac{q^2+q}{2}-q+1 \leq q$. Thus in both cases, we have $\frac{q^2+q}{2} +1 \leq 2q$ which is equivalent to $q^2 - 3q + 2 \leq 0$. The only prime power that satisfies this inequality is $q = 2$, and so we are limited to $\Lnq{3}{2}$, the Fano plane. If $t = 1$, then $\size{\Aa} = \size{\Bb} = 0$ and $\size{\Xx} = \size{\Yy} = 4$, and any such configuration (four points and four lines in the Fano plane) will contain a $\wedge$ or a $\vee$. Hence, $t = 0$, $\size{\Aa} = 1$, $\size{\Bb} = 0$, and $\size{\Xx} = \size{\Yy} = 3$.
\end{proof}

The configuration in Claim \ref{cl:qis2} is exactly the configuration on the right in Figure \ref{fig:FanoExtremes}. Hence, we have shown that if $\size{\Aa} \geq \size{\Bb}$, then either $\Ff = \Aa$ is the collection of all 1-dimensional subspaces or we have the specific configuration in the Fano plane. The proof of Proposition \ref{prop:L3q} is now complete.
\end{proof}

Let $G = (X, \Delta, Y)$ be a $d$-regular bipartite graph ($X \cup Y$ is the set of vertices, $\Delta$ is the set of the edges, and all the edges go from $X$ to $Y$). We can view $G$ as a regular graded poset with two levels. A reasonable question is whether $X$ and $Y$ are the only maximum-sized $(\wedge, \vee)$-free subgraphs of $G$. The example of the incidence graph of the Fano plane (Figure \ref{fig:FanoExtremes}) shows that the answer is no. The example of Figure \ref{fig:bipartiteno} shows that a maximum-sized $(\wedge, \vee)$-free subgraph may even be bigger than $X$ or $Y$.
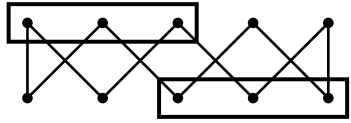
\begin{figure}[ht]
\begin{tikzpicture}
\foreach \x in {1,2,3,4,5}{
\fill (\x,0) circle (2pt);
\fill (\x,1) circle (2pt);
}
\draw[line width=1] (1,0)--(1,1)
(1,0)--(2,1)
(2,0)--(1,1)
(2,0)--(3,1)
(3,0)--(2,1)
(3,0)--(4,1)
(4,0)--(3,1)
(4,0)--(5,1)
(5,0)--(4,1)
(5,0)--(5,1);
\draw[line width=1.5] (.75,1.25) rectangle (3.25,.75)
(2.75,.25) rectangle (5.25,-.25);
\end{tikzpicture}
\caption{Proposition \ref{prop:L3q} does not generalize to general regular bipartite graphs. The boxed elements form a maximum-sized $(\wedge, \vee)$-free subgraph.\label{fig:bipartiteno}}
\end{figure}

If $U$ and $W$ are subspaces in $\Lnq{n}{q}$ with $U < W$, then the interval $[U, W]$ consists of all subspaces that contain $U$ and are contained in $W$. In other words, $[U, W] = \{ V \in \Lnq{n}{q} \mid U \leq V \leq W\}$. If $U$ is of dimension $m$ and $W$ of dimension $n$, then $[U, W]$ is a poset isomorphic to $\Lnq{n-m}{q}$. Before we prove Theorem \hyperlink{thmA}{A}, we gather the (somewhat tedious) argument for a very special case in a preliminary lemma. 

\begin{lem}\label{lem:n5q2}
Let $\Ff$ be a $(\wedge, \vee)$-free family of subspaces in $\Lnq{5}{2}$. Assume that $\size{\Ff} = \gbin{5}{2}_2$, and that the elements of $\Ff$ are either subspaces of dimension $2$ or $3$. Further assume that for arbitrary $U, W \in \Lnq{5}{2}$ with $\dim(U) = 1$, $\dim(W) = 4$, and $U < W$, we have $\size{\Ff \cap [U, W]} = 7$. Then $\Ff$ consists of either all subspaces of dimension $2$ or all subspaces of dimension $3$ in $\Lnq{5}{2}$.
\end{lem}

\begin{proof}
The conclusion is true, by the strict Sperner property of the linear lattices (Theorem \ref{thm:maxantichainLnq}), if $\Ff$ happens to be an anti-chain. So by way of contradiction assume $a, b \in \Ff$ with $a < b$ and $\dim(a) = 2$ and $\dim(b) = 3$. Fix $W \in \Lnq{5}{2}$ with $\dim(W) = 4$ and $b < W$. If $U \in \Lnq{5}{2}$ with $\dim(U) = 1$ and $U \leq W$, then the interval $[U, W]$ is isomorphic to $\Lnq{3}{2}$, and, we have assumed that $\size{\Ff \cap [U, W]} = \gbin{3}{1}_2 = 7$.  As a result $\Ff \cap [U, W]$ is one of the four possibilities spelled out in Proposition \ref{prop:L3q}. We say that $[U, W]$ is of type $(i, j)$ if $i$ and $j$ are respectively the number of subspaces of dimension $2$ and $3$ in $\Ff \cap [U, W]$. By Proposition \ref{prop:L3q}, the only types possible are $(7, 0)$, $(0, 7)$, $(3, 4)$, and $(4,3)$.

Fix $U_0 \in \Lnq{5}{2}$ with $\dim(U_0) = 1$ and $U_0 < a$. Since $U_0 < a < b < W$, and both $a$ and $b$ are in $\Ff$, $[U_0, W]$ is either of type $(3, 4)$ or $(4, 3)$. By symmetry and with no loss of generality, assume $[U_0, W]$ is of type $(3, 4)$. (See Figure \ref{fig:U0W}.)

\begin{figure}[ht]
\begin{tikzpicture}
\node (W) at (2.5,5){};
\draw (W)node[above=3pt]{\small$W$\normalsize} circle (2pt);
\node (b4) at (4,3.5){};
\fill (b4)node[above=3pt]{\small$b_4$\normalsize} circle (2.5pt);
\node (a4) at (1,2){};
\draw (a4)node[below=3pt]{\small$a_4$\normalsize} circle (2pt);
\node (U0) at (2.5,0.5){};
\draw (U0)node[below=3pt]{\small$U_0$\normalsize} circle (2pt);
\foreach \x in {1,2,3}{
\node (b\x) at ({\x/2},3.5){};
\fill (b\x)node[above=3pt]{\small$b_{\x}$\normalsize} circle (2.5pt);
\draw (b\x)--(a4);
}
\foreach \x in {5,6,7}{
\node (b\x) at ({\x/2-.5},3.5){};
\draw (b\x)node[above=3pt]{\small$b_{\x}$\normalsize} circle (2pt);
}
\foreach \x in {1,2,3}{
\node (a\x) at ({\x/2+1.5},2){};
\fill (a\x)node[below=3pt]{\small$a_{\x}$\normalsize} circle (2.5pt);
\draw (a\x)--(b\x);
}
\foreach \x in {5,6,7}{
\node (a\x) at ({\x/2+1},2){};
\draw (a\x)node[below=3pt]{\small$a_{\x}$\normalsize} circle (2pt);
\draw (a\x)--(b\x)
(a\x)--(b4);
}
\draw[thin] (b5)--(a1)
(b5)--(a2)
(b6)--(a2)
(b6)--(a3)
(b7)--(a1)
(b7)--(a3);
\draw[thin] (a5)--(b3)
(a6)--(b1)
(a7)--(b2);
\draw[line width=1.5pt] (W)--(0,3.8)--(0,1.7)--(U0)--(5,1.7)--(5,3.8)--(W);
\foreach \x in {1,2,3,4,5,6,7}{
\draw[thin] (W)--(b\x)
(U0)--(a\x);
}
\foreach \x in {1,2}{
\node (U\x) at ({2.5-\x},.5){};
\draw (U\x)node[below=3pt]{\small$U_{\x}$\normalsize} circle (2pt);
\draw (U\x)--(a4);
}
\node (0) at (2.5,-1){};
\draw (0)node[below=3pt]{\small$\{0\}$\normalsize} circle (2pt);
\foreach \x in {0,1,2}{
\draw[thin] (U\x)--(0);
}
\node (U3) at (3.5,.5){};
\draw (U3)node[below=3pt]{\small$U_3$\normalsize} circle (2pt);
\draw (U3)--(a5);
\draw[line width=1pt] (-2,3.5)--(0,3.5)
(5,3.5)--(7,3.5)
(-4,2)--(0,2)
(5,2)--(9,2)
(-2,.5)--(0,.5)
(5,.5)--(7,.5);
\node at (10,-1){$1$};
\node at (10,.5){$15$};
\node at (10,2){$35$};
\node at (10,3.5){$15$};
\node at (10,5){$1$};
\end{tikzpicture}
\caption{$W$ is a 4-dimensional vector space over $\F_2$, its subspaces form a poset isomorphic to $\Lnq{4}{2}$ with rank numbers $1$, $15$, $35$, $15$, and $1$. $U_0$ is a 1-dimensional subspace, and we have boxed out the interval $[U_0, W] \cong \Lnq{3}{2}$. The seven subspaces identified with a black dot form a maximum-sized $(\wedge, \vee)$-free collection in $[U_0, W]$. \label{fig:U0W}}
\end{figure}
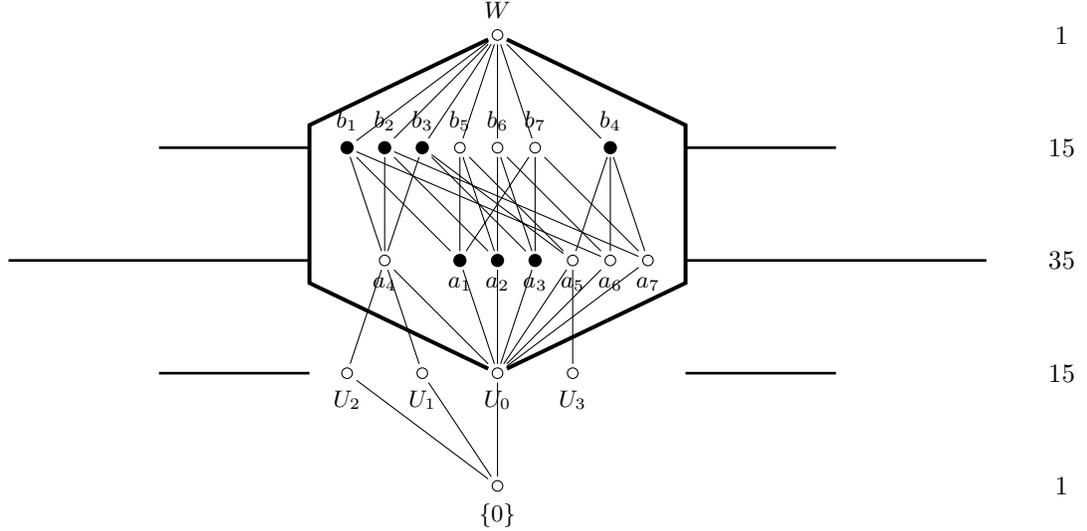 

Thinking of $[U_0, W]$ as the Fano plane, $\Ff \cap [U_0, W]$ is the configuration on the left of Figure \ref{fig:FanoExtremes}. In this configuration, the elements of $\Ff$ are three points $a_1$, $a_2$, and $a_3$ (subspaces of dimension $2$) that are vertices of a triangle, and the four edges (subspaces of dimension 3) that are \emph{not} the edges of that triangle. Among the four edges, one goes through none of the three points, and each of the three others have one of the vertices on them. Let the latter three lines be $b_1$, $b_2$, and $b_3$, and the former be $b_4$. We then have $b_1 \cap b_2  \cap b_3 = a_4 \not\in \Ff$. The interval $[U_0, W]$ has three more lines which are not in $\Ff$ which we call $b_5$, $b_6$, and $b_7$. There are also three more points not in $\Ff$ which we call $a_5$, $a_6$, and $a_7$. (See Figure \ref{fig:U0W}.) 

Focus on the point $a_4$. This is a 2-dimensional subspace in $[U_0, W]$ and has two 1-dimensional subspaces $U_1$ and $U_2$ other than $U_0$.  We claim that, for $i = 1, 2$, $[U_i, W]$ is of type $(0, 7)$. Both intervals contain $a_4 \not\in \Ff$ and so the type cannot be $(7,0)$. They also contain $b_1$, $b_2$, and $b_3$, and these three lines intersect in one point, and this means that the type is not $(4, 3)$ (in type $(4,3)$, the edges form a triangle and do not go through one point). To be of type $(3,4)$, $\Ff \cap [U_i, W]$ needs three lines each with one point of $\Ff$ on them. At least one (actually at least two) of these lines would have to be from among $b_1$, $b_2$, and $b_3$. But the points $a_1$, $a_2$, and $a_3$ are not in $[U_i, W]$ since, for $j = 1, 2, 3$, $a_4 \cap a_j = U_0$ and not $U_1$ or $U_2$. So at least one of $b_1$, $b_2$, or $b_3$ would need a point in $\Ff$ other than the corresponding $a_j$. But this would create a $\wedge$, and so, for $i = 1, 2$,  $[U_i, W]$ cannot be of type $(3, 4)$ either. We conclude that both $[U_1, W]$ and $[U_2, W]$ are of type $(0, 7)$. 
 
Note that if $b$ is any $3$-dimensional subspace in $W$, then $\dim(b \cap a_4) = \dim(b) + \dim(a_4) - \dim(b+a_4) \geq 3+2-4 = 1$ and so $b$ contains one of $U_0$, $U_1$, or $U_2$. Hence, all $3$-dimensional subspaces in $W$ are in one of $[U_i, W]$ for $i = 0, 1, 2$. We conclude that out of the 15 3-dimensional subspaces of $W$, the only ones \emph{not} in $\Ff$ are $b_5$, $b_6$, and $b_7$, and the other 12 3-dimensional subspaces of $W$ are in $\Ff$. 
 
If $a$ is any 2-dimensional subspace in $W$, then $a$ is contained in three $3$-dimensional subspaces of $W$. If $a \in \Ff$, then it could be contained by at most one 3-dimensional subspace in $\Ff$ (since otherwise $\Ff$ would contain a $\vee$). This means that $a$ would have to be contained in at least two of $b_5$, $b_6$, and $b_7$. But these lines form a triangle and their points of intersections are $a_1$, $a_2$, and $a_3$. Hence, the only 2-dimensional subspaces in $W \cap \Ff$ are $a_1$, $a_2$, and $a_3$.  
 
 The 3-dimensional subspace $b_5$ is not in $\Ff$ (it is one of the edges of a triangle in the Fano plane) and has three subspaces of dimension $2$ (points in the Fano plane) that contain $U_0$. Which ones are these? $a_4$ is not one of these since $a_4$ is already on the three lines $b_1$, $b_2$, and $b_3$. The set of the three points cannot be $\{a_1, a_2, a_3\}$ since these points are the corners of a triangle and don't lie on the same line. Hence $b_5$ has a point $a_5$ that is not in $\Ff$. Now let $U_3 \in \Lnq{5}{2}$ be such that $\dim(U_3) = 1$ and $U_3 < a_5$. Then $U_3 < a_5 < b_5 < W$ and neither $a_5$ not $b_5$ are in $\Ff$. Hence $[U_3, W]$ is not of type $(0, 7)$. It is also not of type  $(7, 0)$ or $(4,3)$ because there are only $3$ 2-dimensional subspaces of $W$ that are in $\Ff$. Finally, $[U_3, W]$ cannot be of type $(3,4)$ because if it was, it would have to include $a_1$, $a_2$, and $a_3$. However, $a_5 \cap a_0 = U_0$ and $U_0$ is the only 1-dimensional subspace that is contained in both $a_5$ and $a_1$. Hence $U_3$ is not contained in $a_1$.  This means that $[U_3, W]$ does not have a valid type which is a contradiction. The contradiction proves that $\Ff$ is an anti-chain after all and the proof is complete. 
\end{proof}

We are now ready to prove Theorem \hyperlink{thmA}{A}.

\begin{proof}[\textbf{Proof of Theorem \hyperlink{thmA}{A}}]
When $n$ is even, the result is a special case of Theorem \hyperlink{thmB}{B}. So let $n$ be odd and proceed by induction on $n$. The base case $n = 3$ was proved in Proposition \ref{prop:L3q} and so let $n = 2k+1 > 3$ and assume that the result is true for all smaller $n$. We first prove that $\mathrm{ex}(\Lnq{n}{q}; \wedge,\vee) = \gbin{n}{k}_q$. Note that $\floor{n/2} = k$ and the collection of all subspaces of dimension $k$ (as well as the collection of those of dimension $k+1$) does not contain a $\wedge$ or a $\vee$. Hence,  $\mathrm{ex}(\Lnq{n}{q}; \wedge,\vee) \geq \gbin{n}{k}_q$. Now, let $\Ff$ be a family of subspaces that does not contain a $\wedge$ or a $\vee$. By Corollary \ref{cor:Lnqbroomfork1}, to show that $\size{\Ff} \leq \gbin{n}{k}_q$, we can assume that $\Ff$ consists only of subspaces of dimensions $k$ or $k+1$.

If $U, W \in \Lnq{n}{q}$ with $\dim(U) = 1$ and $\dim(W) = n-1$, then the interval $[U, W] = \{ V \in \Lnq{n}{q} \mid U \subseteq V \subseteq W\}$ is a poset isomorphic to $\Lnq{n-2}{q}$---the poset $[U,W]$ is isomorphic to the poset of subspaces of $W/U$---and so by induction $\size{\Ff \cap [U, W]} \leq \gbin{n-2}{k-1}_q$. Now consider
$$\Ss = \{ (V, [U, W]) \mid V \in \Ff, U \subseteq V \subseteq W, \dim(U)=1, \dim(W) = n-1\}.$$
Note that in $\Lnq{n}{q}$, there are $\gbin{n}{n-1}_q = [n]_q$ subspaces of dimension $n-1$ and each of these contains $\gbin{n-1}{1}_q = [n-1]_q$ subspaces of dimension $1$. Hence, there are a total of $[n]_q [n-1]_q$ intervals $[U, W]$ with $\dim(U) = 1$ and $\dim(W) = n-1$. Likewise, any subspace of dimension $k$ (or $k+1$) contains $[k]_q$ (or $[k+1]_q$) subspaces of dimension $1$ and is contained in $[k+1]_q$(or $[k]_q$) subspaces of dimension $n-1$. Hence, every subspace of dimension $k$ or $k+1$ is contained in $[k]_q [k+1]_q$ intervals $[U,W]$ with $\dim(U) =1$ and $\dim(W) = n-1$. Now, counting the size of $\Ss$ in two different ways we get
$$\size{\Ff} [k]_q [k+1]_q = \size{\Ss} = \sum_{\substack{U \subseteq W\\ \dim (U) = 1 \\ \dim (W) = n-1}} \size{\Ff \cap [U,W]} \leq [n]_q [n-1]_q \gbin{n-2}{k-1}_q.$$
Hence, 
$$\size{\Ff} \leq \frac{[n]_q [n-1]_q [n-2]_q!}{[k+1]_q[k]_q[k-1]_q! [k]_q!} = \frac{[n]_q!}{[k]_q! [k+1]_q!} = \gbin{n}{k}_q,$$
and we have proved that $\mathrm{ex}(\Lnq{n}{q}; \wedge,\vee) = \gbin{n}{k}_q$.

Now let $\Ff$ be a maximum-sized $(\wedge, \vee)$-free family of subspaces in $\Lnq{n}{q}$. Initially, we assume that every element of $\Ff$ is of dimension $k$ or $k+1$. In the calculation above for $\size{\Ss}$, we must have equality throughout and so $\size{\Ff \cap [U, W]} = \gbin{n-2}{k-1}_q$ for every interval $[U, W]$ where $U \subseteq W$ and $\dim(U) = 1$ and $\dim(W) = n-1$. Since $\Ff \cap [U, W]$ is $(\wedge, \vee)$-free, and as long as $n-2 > 3$ or $q > 2$, by induction all elements of $\Ff \cap [U, W]$ have rank $k$ (in $\Lnq{n}{q}$), or all elements have rank $k+1$. We claim that this means that $\Ff$ is an anti-chain. This is because if $a, b \in \Ff$ and $a < b$, then both $a$ and $b$ are contained in the same interval $[U, W]$ (take $U$ be a 1-dimensional subspace of $a$ and $W$ an $n-1$ dimensional subspace that contains $b$), and we just proved that the subspaces in $\Ff \cap [U, W]$ are all of the same dimension. Hence, $\Ff$ is an anti-chain of size $\gbin{n}{k}_q$ and by the strict Sperner property (Theorem \ref{thm:maxantichainLnq}), $\Ff$ is one of the middle levels of $\Lnq{n}{q}$. The only case remaining (given the additional assumption that $\Ff$ is limited to the middle two levels) is when $n = 5$ and $q = 2$, and this was proved in Lemma \ref{lem:n5q2}.

We now complete the proof by showing that a maximum-sized $(\wedge, \vee)$-free family $\Ff$ must be limited to the middle two levels. Assume $\Ff$ is a $(\wedge, \vee)$-free family of subspaces of $\Lnq{n}{q}$ of size $\gbin{n}{k}_q$, and assume $\Ff$ has elements outside of the two middle levels, then we can use Proposition \ref{prop:Pushdown} and/or Proposition \ref{prop:Pushup} to find $\Ff^\prime$ a $(\wedge, \vee)$-free family of subspaces of exactly the same size that is limited to the middle two levels. These propositions replace all subspaces of dimension greater than $\ceil{\frac{n}{2}}$ with an equal number of subspaces of dimension $\ceil{\frac{n}{2}}$. Likewise, all subspaces of dimension less than $\floor{\frac{n}{2}}$ are replaced with an equal number of subspaces of dimension $\floor{\frac{n}{2}}$. If $\Ff$ originally contained subspaces of dimensions more than $n/2$ as well as subspaces with dimensions less than $n/2$, then the resulting family $\Ff^\prime$ would have subspaces of dimensions $\floor{n/2}$ as well as $\ceil{n/2}$. But we just proved that such a maximum-sized family does not exist. Hence, without loss of generality, assume that the dimension of all subspaces in $\Ff$ is greater than $n/2$. Now use Proposition \ref{prop:Pushdown} to construct a maximum-sized $(\wedge, \vee)$-free family $\Ff^\prime$ that is confined to levels $\ceil{n/2}$ and $\ceil{n/2} + 1$. (In other words, stop one step before pushing all elements of $\Ff$ into level $\ceil{n/2}$). At this point, apply Proposition \ref{prop:NMbroomfree} (in the parlance of that Proposition $\alpha > 0$) to get a larger family that is confined to the level $\ceil{n/2}$. This family will certainly be $(\wedge, \vee)$-free since it is confined to one level. This proves that the original family could not have been maximum-sized if it originally contained any elements outside of the middle two levels. This completes the proof.
\end{proof}

\begin{rem}
While it was convenient to organize the proof by using the fact that the only maximum-sized anti-chains in $\Lnq{n}{q}$ are $\bc{\Lnq{n}{q}}{\floor{n/2}}$ and  $\bc{\Lnq{n}{q}}{\ceil{n/2}}$, this wasn't strictly necessary. It is not too difficult to provide direct arguments thereby providing an alternative proof for Theorem \ref{thm:maxantichainLnq}, the characterization of maximum anti-chains in $\Lnq{n}{q}$.
\end{rem}

%%%%%%%%%%%%%%%%%%%%%%%%
\section{Butterflies, Y, and Y'}
%%%%%%%%%%%%%%%%%%%%%

To get an upper bound for the maximal size of a butterfly-free or (Y, Y')-free family in $\Lnq{n}{q}$, we will partition any such family $\Ff$ into anti-chains, and then use the familiar technique for proving LYM inequalities by counting maximal chains of $\Lnq{n}{q}$ that pass through these. This is the strategy of De Bonis et al. \cite{DeBonisKatonaSwa:05} in proving Theorem \ref{thm:nobutterflyboolean}, and our proof of the preparatory inequalities (Lemma \ref{lem:twoantichains} and Proposition \ref{prop:LYMtype}) follows their proof, with modifications for the linear lattice case.

\begin{lem}\label{lem:twoantichains} 
Let $n$ be a positive integer, and $q$ a power of a prime. Let $\Mm, \Aa \subseteq \Lnq{n}{q} - \{\{0\}, (\F_q)^n\}$ be two disjoint anti-chains such that either for all $A \in \Aa$, there exists a unique $M \in \Mm$ such that $A < M$, or for all $A \in \Aa$, there exists a unique $M \in \Mm$ such that $A > M$.
Then
$$\sum_{M \in \Mm} {n \brack \dim M}_q^{-1} + \frac{q}{q+1} \sum_{A \in \Aa} {n \brack \dim A}_q^{-1} \leq 1.$$
\end{lem}

\begin{proof}
For $A \in \Aa$, denote by $f(A)$ the unique element of $\Mm$ comparable to $A$.  We prove the case where $A < f(A)$, for all $A \in \Aa$ (the argument for the other case is identical).  The total number of maximal chains in $\Lnq{n}{q}$ is $[n]_q!$ and, if $U, V \in \Lnq{n}{q}$ with $U \leq V$, then the interval $[U, V] = \{W \in \Lnq{n}{q} \mid U \leq W \leq V\}$ is isomorphic as a poset to the linear lattice of dimension $\dim V - \dim U$, and so has $[\dim V - \dim U]_q!$ maximal chains. As a result, for every $V \in \Lnq{n}{q}$, there are $[\dim V]_q!\ [n-\dim V]_q!$ maximal chains going through $V$. Since $\Aa$ and $\Mm$ are anti-chains, the total number of maximal chains that go through $\Aa \cup \Mm$ are the sum of the number of maximal chains that go through each element of $\Aa$ and $\Mm$ minus the number of such chains that go through one element in $\Aa$ and one element in $\Mm$. For $A \in \Aa$ and $f(A) \in \Mm$, the number of maximal chains of $\Lnq{n}{q}$ that go through both $A$ and $f(A)$ is
$$[\dim A]_q!\ [\dim f(A) - \dim A]_q!\ [n-\dim f(A)]_q!.$$
Hence, we have
\begin{align*}
[n]_q ! & \geq \sum_{A \in \Aa} [\dim A]_q!\ [n-\dim A]_q! + \sum_{M \in \Mm} [\dim M]_q!\ [n-\dim M]_q! \\
& - \sum_{A \in \Aa} [\dim A]_q!\ [\dim f(A) - \dim A]_q!\ [n-\dim f(A)]_q!.
\end{align*}
Dividing both sides by $[n]_q!$ and rearranging gives
$$1 \geq \sum_{M \in \Mm} {n \brack \dim M}_q^{-1} + \sum_{A \in \Aa} {n \brack \dim A}_q^{-1}\left(1-{n- \dim A \brack n- \dim f(A)}_q^{-1} \right). $$
Since $(\F_q)^n \not \in \Mm$, for all $A \in \Aa$ we have $\dim A < \dim f(A) < n$; in particular, $n - \dim A \geq 2$. Thus
$${n- \dim A \brack n- \dim f(A)}_q \geq [n-\dim A]_q \geq [2]_q =q+1.$$
Hence our desired inequality follows.
\end{proof}

We now prove an LYM-type inequality for $(\textrm{Y}, \textrm{Y'})$-free families in $\Lnq{n}{q}$.

\begin{prop}\label{prop:LYMtype}
Let $n \geq 3$ be an integer, $q$ a power of a prime, and let $\Ff$ be a $(\textrm{Y}, \textrm{Y'})$-free subposet of $\Lnq{n}{q}-\{\{0\}, (\F_q)^n\}$. Then
\begin{equation}\sum_{V \in \Ff} {n \brack \dim V}_q^{-1} \leq 2.\label{eq:LYMtype}\end{equation}
Furthermore, if we have equality in \eqref{eq:LYMtype}, then $\Ff$ is the union of its maximal and minimal elements.
\end{prop}

\begin{proof}
Let $\Mm_1$ be the set of maximal elements of $\Ff$, $\Mm_2$ the set of minimal elements of $\Ff-\Mm_1$, and $\Aa = \Ff - (\Mm_1 \cup \Mm_2)$. The families $\Mm_1$, $\Mm_2$ are anti-chains, and, by their definitions, for each $A \in \Aa$ there exists elements $M_1 \in \Mm_1$ and $M_2 \in \Mm_2$ such that $M_1 < A < M_2$. Moreover, since $\Ff$ is  $(\textrm{Y}, \textrm{Y'})$-free, $M_1$ and $M_2$ are unique. The family $\Aa$ is also an anti-chain since otherwise $\Ff$ would include a chain of length $3$, and would not be $\textrm{Y}$-free. 

We now apply Lemma \ref{lem:twoantichains} twice, once to $\Mm_1$ and $\Aa$ and once to $\Mm_2$ and $\Aa$. For $i = 1, 2$, we have
$$\sum_{M \in \Mm_i} \gbin{n}{\dim M}_q^{-1} + \frac{q}{q+1} \sum_{A \in \Aa} \gbin{n}{\dim A}_q^{-1} \leq 1.$$
Adding the two inequalities for $i = 1$ and $2$, we get
$$\sum_{M \in \Mm_1} \gbin{n}{\dim M}_q^{-1} + \sum_{M \in \Mm_2} \gbin{n}{\dim M}_q^{-1} + \frac{2q}{q+1} \sum_{A \in \Aa} \gbin{n}{\dim A}_q^{-1} \leq 2.$$
Now since $\Mm_1$, $\Mm_2$, and $\Aa$ partition $\Ff$, and $1 - \frac{2q}{q+1} = \frac{q-1}{q+1} > 0$, we have
$$\sum_{V \in \Ff} \gbin{n}{\dim V}_q^{-1} + \frac{q-1}{q+1} \sum_{A \in \Aa} \gbin{n}{\dim A}_q^{-1} \leq 2.$$
As a result \eqref{eq:LYMtype} follows, and, in the case of equality, $\Aa = \emptyset$.
\end{proof}

We also need the following straightforward lemma.
\begin{lem}[Thanh, 1998\cite{Thanh:98}]\label{lem:maximal}
Let  $m$ be a positive integer, $\delta$ a nonnegative real number, and $S = \{x_j\}_{j \in J}$ a finite multiset of nonnegative real numbers indexed by $J$. Then
$$\max\{|I| \mid I \subseteq J, \sum_{i \in I} x_i \leq \delta\} = K_\delta,$$
where $K_\delta$ is the largest integer (not exceeding $|J|$) such that the sum of the $K_\delta$ smallest elements in $S$ does not exceed $\delta$.
\end{lem}

Note that an index set $I$ for the above lemma is maximal if and only if the elements indexed by $I$ form a set of $K_\delta$ smallest numbers of $S$. This observation will be useful when we characterize our maximal family $\Ff$. We are now ready to prove Theorem \hyperlink{thmC}{C}.

\begin{proof}[\textbf{Proof of Theorem \hyperlink{thmC}{C}}] The intersection of two subspaces is a unique subspace, and so the union of any two consecutive levels of $\Lnq{n}{q}$ is butterfly-free. Also, any butterfly-free family is $(\textrm{Y}, \textrm{Y'})$-free. Hence,
$$\mathrm{ex}(\Lnq{n}{q}; \Yposet, \Ypposet) \geq \mathrm{ex}(\Lnq{n}{q}; \bfly) \geq \gbin{n}{\floor{n/2}}_q + \gbin{n}{\floor{n/2} +1}_q.$$
To establish equalities througout, we now show that 
$$\mathrm{ex}(\Lnq{n}{q};\Yposet, \Ypposet) \leq \gbin{n}{\floor{n/2}}_q + \gbin{n}{\floor{n/2} +1}_q.$$
Let $\Ff$ be a maximum-sized $(\textrm{Y}, \textrm{Y'})$-free family in $\Lnq{n}{q}$. We first show that it suffices to assume that $\Ff$ does not contain $\{0\}$ or $(\F_q)^n$. Indeed, assume that $\{0\} \in \Ff$. If there was a 1-dimensional subspace that was not in $\Ff$, then we could replace $\{0\}$ with this subspace and continue to have a $(\textrm{Y}, \textrm{Y'})$-free family of the same size as $\Ff$. On the other hand, if $\Ff$ contained $\{0\}$ and \emph{every} 1-dimensional subspace, then, since $\Ff$ is $\textrm{Y}$-free, it can contain at most one subspace containing each of the 1-dimensional subspaces. This would mean that the $\size{\Ff} \leq 2[n]_q+1$. For $n \geq 4$, $\gbin{n}{\floor{n/2}}_q + \gbin{n}{\floor{n/2} +1}_q$, the sum of the sizes of the two largest level in $\Lnq{n}{q}$, is bigger than $2[n]_q + 1$, contradicting the maximality of the size of $\Ff$. The remaining case is $n = 3$, and $\Ff$ contains $\{0\}$ and every 1-dimensional subspace. If $\Ff$ also contains $(\F_q)^3$, then it cannot contain any other subspace (since otherwise $\Ff$ would contain a $Y^\prime$). But now $\size{\Ff} = [3]_q+2 < 2[3]_q = \gbin{3}{1}_q + \gbin{3}{2}_q$, a contradiction. If $\Ff$ does not contain $(\F_q)^3$, then $\Ff$ cannot contain all 2-dimensional subspaces, and we can replace $\{0\}$ with a 2-dimensional subspace not in $\Ff$ and get a  $(\textrm{Y}, \textrm{Y'})$-free family of the same size as $\Ff$. Thus, without loss of generality we can assume that $\{0\}$ is not in $\Ff$. The argument for $(\F_q)^n$ would be identical.

Proposition \ref{prop:LYMtype} now applies and we have inequality \eqref{eq:LYMtype}. Apply Lemma \ref{lem:maximal} with the multiset $S=\left\{\gbin{n}{\dim V}_q^{-1} \right\}_{V \in \Lnq{n}{q}}$ and $\delta=2$. 
If $n$ is odd, then $\gbin{n}{\floor{n/2}}_q$ and $\gbin{n}{\floor{n/2}+1}_q$ are the two largest rank numbers; if $n$ is even, then $\gbin{n}{n/2}_q$ is the single largest rank number, and $\gbin{n}{n/2+1}_q = \gbin{n}{n/2-1}_q$ are the two second largest rank numbers. Thus, we see that $K_\delta$ as defined in Lemma \ref{lem:maximal} is equal to $\gbin{n}{\floor{n/2}}_q + \gbin{n}{\floor{n/2}+1}_q$. Since $\Ff$ indexes a subset of numbers in $S$ whose sum does not exceed $\delta = 2$, we deduce that
\begin{equation}\label{eq:upperbound} \size{\Ff} \leq K_\delta = \gbin{n}{\floor{n/2}}_q + \gbin{n}{\floor{n/2}+1}_q,\end{equation}
which gives us the desired upper bound for $\mathrm{ex}(\Lnq{n}{q}; \Yposet, \Ypposet)$.

%%%%%%%%%%%%%%%%%%%%%%%
It remains to show that the only families achieving the bounds are the union of two of the largest levels in the linear lattice. For a maximum-sized family, both inequalities \eqref{eq:LYMtype} and \eqref{eq:upperbound} must be equality. In particular, the equality for (\ref{eq:upperbound}) means that the family $\Ff$ indexes a subset of $K_\delta$ smallest elements in $S$. If $n$ is odd, then $\bc{\Lnq{n}{q}}{\floor{n/2}} \cup \bc{\Lnq{n}{q}}{\floor{n/2}+1}$ is the unique such family.

If $n$ is even, then $\Ff$ must consist of $\bc{\Lnq{n}{q}}{n/2}$ and $\gbin{n}{n/2-1}_q$ elements from $P = \bc{\Lnq{n}{q}}{n/2+1} \cup \bc{\Lnq{n}{q}}{n/2-1}$. We have to show that $\Ff$ cannot pick and choose and must contain either \emph{all} subspaces of dimension $n/2-1$ or all the subspaces of dimension $n/2+1$. View $P$ as a subposet of $\Lnq{n}{q}$, and denote $\Ff' = \Ff - \bc{\Lnq{n}{q}}{n/2} \subseteq P$.

Suppose there exists $X, Y \in \Ff'$ such that $Y < X$. Then $X \in \bc{\Lnq{n}{q}}{n/2+1}$, $Y \in \bc{\Lnq{n}{q}}{n/2-1}$, and we can find $A \in \bc{\Lnq{n}{q}}{n/2} \subset \Ff$ such that $Y < A < X$. But this cannot happen, since the case of equality for (\ref{eq:LYMtype}) requires that every element of $\Ff$ is either maximal or minimal in $\Ff$. Hence $\Ff'$ is an anti-chain in $P$.

As a graded poset, $P$ has height 2 and width $\gbin{n}{n/2-1}_q = \size{\Ff^\prime}$. By Dilworth's Theorem (see Anderson \cite[Theorem 3.2.1]{Anderson:02}), $P$ can be partitioned into $\gbin{n}{n/2-1}_q$ chains of length $1$. In fact, if $X < Y$ are two elements of $P$, we can always arrange the chain $X < Y$ to be one of the chains in this chain partition of $P$. Now $\Ff^\prime$ is an anti-chain with as many elements as chains in the partition. Hence, it has to contain exactly one of $X$ or $Y$. 

Now assume that $\Ff^\prime$ contained a subspace $A$ of dimension $n/2-1$,  and let $B$ be an arbitrary subspace of dimension $n/2+1$. Since the Hasse diagram of $P$ (viewed as a bipartite graph) is connected, we can find a sequence of subspaces $A = A_1, B_1, A_2, B_2, \ldots, A_n, B_n = B$ in $P$, such that, for $1 \leq i \leq n$, $A_i < B_i$ and, for $1 \leq i \leq n-1$, $B_i > A_{i+1}$. Our previous observation says that since $A_1 \in \Ff^\prime$ and $A_1 < B_1$, then $B_1 \not\in \Ff^\prime$. Continuing along the sequence of subspaces, we get that $A_1$, $A_2$, $\ldots$, $A_n$ must be in $\Ff^\prime$ while $B_1$, $B_2$, $\ldots$, $B_n = B$ cannot be in $\Ff^\prime$. We conclude that all subspaces in $\Ff^\prime$ have the same dimension and the proof is complete.
\end{proof}

\section{A conjecture}

Let $Y_k$ denote a poset with $3+k$ elements $a$, $b$, $c_0$, $c_1$, $\ldots$, $c_k$, and with relation $c_k < c_{k-1} < \ldots < c_1 < c_0 < a$ as well as $c_0 < b$. Note that $Y_0$ is just a $\vee$ and $Y_1$ is just $\textrm{Y}$. Likewise, define $Y_k^\prime$ to have the same elements as $Y_k$ but all the relations reversed. Hence, $Y_0^\prime$ is $\wedge$ and $Y_1^\prime$ is $\textrm{Y}^\prime$. Our Theorems \hyperlink{thmA}{A} and \hyperlink{thmC}{C} show that 
\begin{align*} \mathrm{ex}(\Lnq{n}{q}; Y_0, Y_0^\prime) & = \mathrm{ex}(\Lnq{n}{q};\ \twochain\ ) \\
\mathrm{ex}(\Lnq{n}{q}; Y_1, Y_1^\prime) & = \mathrm{ex}(\Lnq{n}{q};\ \threechain\ ).
\end{align*}
Let $P_k$ be a chain of length $k$ (and size $k+1$), then we conjecture that 
$$\mathrm{ex}(\Lnq{n}{q}; Y_k, Y_k^\prime)  = \mathrm{ex}(\Lnq{n}{q};\  P_{k+1}).$$
%%%%%%%%%%%%%%%%%%%%%%%%%%%%%%%%%%%%%%%%%%%%%%%%%%%%%%%%%%%%%%%%%%%%%%%%
%%                  BIBLIOGRAPHY
%%%%%%%%%%%%%%%%%%%%%%%%%%%%%%%%%%%%%%%%%%%%%%%%%%%%%%%%%%%%%%%%%%%%%%%%
\newcommand{\ugst}{\relax}
\bibliographystyle{amsplain}

\begin{thebibliography}{10}

\bibitem{Anderson:02}
Ian Anderson, \emph{Combinatorics of finite sets}, Dover Publications, Mineola,
  NY, 2002, Corrected reprint of the 1989 edition published by {O}xford
  {U}niversity {P}ress, Oxford.

\bibitem{Baker:69}
Kirby~A. Baker, \emph{A generalization of {S}perner's lemma}, J. Combinatorial
  Theory Ser. A \textbf{6} (1969), no.~2, 224--225.

\bibitem{DeBonisKatonaSwa:05}
Annalisa De~Bonis, Gyula O.~H. Katona, and Konrad~J. Swanepoel, \emph{Largest
  family without {$A\cup B\subseteq C\cap D$}}, J. Combin. Theory Ser. A
  \textbf{111} (2005), no.~2, 331--336. 

\bibitem{Engel:97}
Konrad Engel, \emph{Sperner theory}, Encyclopedia of Mathematics, Cambridge
  University Press, Cambridge, 1997.

\bibitem{GrahamHar:69}
R.~L. Graham and L.~H. Harper, \emph{Some results on matching in bipartite
  graphs}, SIAM J. Appl. Math. \textbf{17} (1969), 1017--1022.

\bibitem{GriggsLiLu:12}
Jerrold~R. Griggs, Wei-Tian Li, and Linyuan Lu, \emph{Diamond-free families},
  J. Combin. Theory Ser. A \textbf{119} (2012), no.~2, 310--322. 
  
\bibitem{GroszMetTomp:18}
D\'aniel Gr\'osz, Abhishek Methuku, and Casey Tompkins, \emph{An upper bound on
  the size of diamond-free families of sets}, J. Combin. Theory Ser. A
  \textbf{156} (2018), 164--194.

\bibitem{Katona:08}
Gyula O.~H. Katona, \emph{Forbidden intersection patterns in the families of
  subsets (introducing a method)}, Horizons of combinatorics, Bolyai Soc. Math.
  Stud., vol.~17, Springer, Berlin, 2008, pp.~119--140. 

\bibitem{KatonaTarjan:83}
Gyula O.~H. Katona and Tam\'{a}s~G. Tarj{\'a}n, \emph{Extremal problems with
  excluded subgraphs in the {$n$}-cube}, Graph theory (\L ag\'ow, 1981),
  Lecture Notes in Math., vol. 1018, Springer, Berlin, 1983, pp.~84--93.


\bibitem{Kleitman:74}
D.~J. Kleitman, \emph{On an extremal property of antichains in partial orders.
  {T}he {${\rm LYM}$} property and some of its implications and applications},
  Combinatorics (Proc. NATO Advanced Study Inst., Breukelen, 1974), Part 2:
  Graph theory; foundations, partitions and combinatorial geometry, Math.
  Centrum, Amsterdam, 1974, pp.~77--90. Math. Centre Tracts, No. 56.

\bibitem{Nagy:18}
D\'aniel~T. Nagy, \emph{Forbidden subposet problems with size restrictions}, J.
  Combin. Theory Ser. A \textbf{155} (2018), 42--66. 

\bibitem{Salerno:09}
Paul Salerno, \emph{{Forbidden Posets}}, Undergraduate thesis, Pomona College,
  Claremont, CA, USA, 2009.

\bibitem{PCURC-S11:14}
Ghassan Sarkis, Shahriar Shahriari, and PCURC, \emph{Diamond-free subsets in
  the linear lattices}, Order \textbf{31} (2014), no.~3, 421--433 (English),
  PCURC stands for the Pomona College Undergraduate Research Circle whose
  members in Spring of 2011 were Zachary Barnett\ugst, David Breese\ugst,
  Benjamin Fish\ugst, William Frick\ugst, Andrew Khatutsky\ugst, Daniel
  McGuinness\ugst, Dustin Rodrigues\ugst, and Claire Ruberman\ugst.

\bibitem{Sperner:28}
Emanuel Sperner, \emph{Ein satz {\"{u}ber} {U}ntermengen einer endlichen
  {M}engen}, Math. Z. \textbf{27} (1928), 544--8.

\bibitem{Thanh:98}
Hai~Tran Thanh, \emph{An extremal problem with excluded subposet in the
  {B}oolean lattice}, Order \textbf{15} (1998), no.~1, 51--57. 
  
\bibitem{vanLintWil:01}
J.~H. van Lint and Richard~M. Wilson, \emph{A course in combinatorics}, second
  ed., Cambridge University Press, Cambridge, 2001.
\end{thebibliography}

\providecommand{\bysame}{\leavevmode\hbox to3em{\hrulefill}\thinspace}
\providecommand{\MR}{\relax\ifhmode\unskip\space\fi MR }

\providecommand{\MRhref}[2]{%
  \href{http://www.ams.org/mathscinet-getitem?mr=#1}{#2}
}
\providecommand{\href}[2]{#2}

\end{document}